\documentclass{amsart}
\usepackage{amsfonts}
\usepackage{color}
\usepackage{graphicx}
\usepackage[a4paper,bookmarksnumbered,colorlinks, linkcolor=blue, citecolor=red, pagebackref, bookmarks, breaklinks]{hyperref}

\newtheorem{thm}{Theorem}[section]
\newtheorem{cor}[thm]{Corollary}

\theoremstyle{definition}

\theoremstyle{remark}

\newtheorem{conj}[thm]{Conjecture}
\numberwithin{equation}{section}
\newcommand{\R}{\mathbb R}
\newcommand{\N}{\mathbb N}
 
\newcommand{\ve}{\varepsilon}
\newcommand{\lam}{\lambda}
\newcommand{\cde}{\stackrel{*}{\rightharpoonup}}

\newcommand{\HH}{\mathcal{H}^{n-1}}

\def\diver{\mathop{\text{\normalfont div}}}
\title[Homogenization of Steklov eigenvalues]{Homogenization of Steklov eigenvalues with rapidly oscillating weights}
\author{Ariel M.  Salort}%

\address{Departamento de Matem\'atica, FCEyN - Universidad de Buenos Aires and
\hfill\break \indent IMAS - CONICET
\hfill\break \indent Ciudad Universitaria, Pabell\'on I (1428) Av. Cantilo s/n. \hfill\break \indent Buenos Aires, Argentina.}

\email[A.M. Salort]{asalort@dm.uba.ar}
\urladdr{http://mate.dm.uba.ar/~asalort}

\begin{document}



\begin{abstract}
 In this article we study the homogenization rates of eigenvalues of a Steklov problem with rapidly oscillating periodic weight functions. The results are obtained via a careful study of oscillating functions on the boundary and a precise estimate of the $L^\infty$ bound of eigenfunctions. As an application we provide some estimates on the first nontrivial curve of the Dancer-{F}u{\v{c}}{\'{\i}}k spectrum.
\end{abstract}

\maketitle
 
\section{Introduction}
Homogenization of elliptic operators with rapidly oscillating coefficients  has been and continues to be a very active research area due to the wide applications. We refer to \cite{A,CZ,JKO,k1,k2,MV,  OSY,SV} for an introduction and a comprehensive develop of this theory.

In particular, homogenization of eigenvalue problems involving rapidly oscillating quasilinear operators and/or rapidly oscillating weight functions is a field which has received great attention in the last decades. Denoting by $\lam_{k,\ve}$ the $k-$th (variational) eigenvalue corresponding to a  rapidly oscillating equation, and $\lam_{k,0}$ the $k-$th (variational) eigenvalue of the corresponding limit problem as $\ve \to 0$, $k\in \N$, the labor of obtaining estimates of the difference $|\lam_{k,\ve}-\lam_{k,0}|$ in terms of $k$ and $\ve$ was addressed by several authors: the order of convergence of eigenvalues in the Dirichlet/Neumann case was treated by Kesavan \cite{k1,k2}, Osborn \cite{Os}, Oleinik \emph{et al} \cite{OSY}, Vogelius \emph{et al} \cite{MV, SV}, Castro and Zuazua \cite{CZ}, Kenig, Lin and Shen \cite{KLS}, Fern\'andez Bonder \emph{et al} \cite{FBPS1, FBPS2, FBPS3, FBPS4}, among others. See also \cite{S1,S2}.

Nevertheless, to the best of our knowledge, no investigation were performed on convergence rates of eigenvalues involving quasilinear problems with Steklov boundary condition and rapidly oscillating weight functions. Inspired on \cite{ASS, CK}, that is the main scope of this manuscript.

Given $\Omega\subset \R^n$, $n\geq 3$  and an open convex bounded domain with Lipschitz boundary we consider the Steklov eigenvalue problem for the $p-$Laplacian
\begin{align} \label{ec.stek.intro} \tag{$P_\rho$}
\begin{cases}
-\Delta_p u + |u|^{p-2}u =0 &\text{ in }\Omega,\\
|\nabla u|^{p-2} \frac{\partial u}{\partial \nu}=   \lambda \rho |u|^{p-2}u &\text{ on } \partial\Omega,
\end{cases}
\end{align}
where $\frac{\partial u}{\partial \nu}$ is the outer normal derivative, $\Delta_p:=\diver(|\nabla u|^{p-2}\nabla u)$, $p>1$,  and $\rho\colon \partial\Omega \to \R$ is a function such that, for some fixed constants $\rho_\pm$ fulfills
\begin{equation} \label{cond.rho}
0<\rho_- \leq \rho(x) \leq \rho_+<\infty, \qquad x\in \partial\Omega.
\end{equation}

Given a sequence of functions $\{\rho_\ve\}_{\ve>0}$  satisfying \eqref{cond.rho}, we  want to study the convergence of eigenvalues of $(P_{\rho_\ve})$ as $\ve\to 0$ to the limit problem $(P_{\rho_0})$, where these two problems are defined, respectively as
\begin{align}  \label{problems}
\begin{cases}
-\Delta_p u_\ve + |u_\ve|^{p-2}u_\ve =0 &\text{ in }\Omega,\\
|\nabla u_\ve|^{p-2} \frac{\partial u_\ve}{\partial \nu}=   \lambda \rho_\ve |u_\ve|^{p-2}u_\ve &\text{ on } \partial\Omega,
\end{cases}
\quad 
\begin{cases}
-\Delta_p u_0 + |u_0|^{p-2}u_0 =0 &\text{ in }\Omega,\\
|\nabla u_0|^{p-2} \frac{\partial u_0}{\partial \nu}=   \lambda \rho_0 |u_0|^{p-2}u_0 &\text{ on } \partial\Omega.
\end{cases}
\end{align}

Under the aforementioned assumptions of  $\{\rho_\ve\}_{\ve>0}$  it is well-known that sequences of eigenpairs $\{(u_\ve, \lam_\ve)\}_{\ve>0}$ solving $(P_{\rho_\ve})$ converge as $\ve\to 0$ to eigenpairs $(u_0,\lam_0)$ to $(P_{\rho_0})$, where $\rho_0$ is the weak* limit in $L^\infty(\partial \Omega)$ of $\rho_\ve$. See for instance \cite{ASS, CK}. 

The main goal of this paper is to study the behavior of the (variational) eigenvalues to $(P_{\rho_\ve})$ as  $\ve\to 0$. When no periodicity assumptions on the family $\{\rho_\ve\}_{\ve>0}$ are made, in Theorem \ref{main} we prove that
$$
\lim_{k\to 0} \lam_{k,\ve} = \lam_{k,0}
$$
where $\lam_{k,\ve}$ and $\lam_{k,0}$ denote the $k-$th ($k\in\N$) variational eigenvalue of problems \eqref{problems}, respectively.

In the case of periodic homogenization, i.e.,  $\rho_\ve(x)=\rho(\tfrac{x}{\ve})$ where $\rho$ is a $Q-$periodic function (being $Q$ the unit cube in $\R^n$): $\rho(x+h)=\rho(x)$ for all $x\in \partial\Omega$ and $h\in\mathbb{Z}^n$, it is well-known that $\rho_\ve \cde \rho_0$ weakly* in $L^\infty(\partial\Omega)$ with $\rho_0=\int_{\mathbb{T}^n} \rho(x)\,dx$, being $\mathbb{T}^n$ the unit torus of $\R^n$. In this case we obtain explicit estimates of the  convergence rate for the first two eigenvalues. Namely, in Theorem \ref{main} it is proved that
$$
|\lam_{k,\ve}-\lam_{k,0}|\leq C   \ve^{\frac{p-1}{p}-\tau}, \qquad k=1,2
$$
for any $\tau>0$, where $C$ is a computable positive constant depending only on $n$, $p$, $\tau$, $\rho_\pm$ and $\Omega$.

In the linear case, that is, when $p=2$, we are able to compute the rate  of the convergence of the full sequence of variational eigenvalues. Namely, in Theorem \ref{teo.p2} we prove that for all eigenvalues the following estimate holds
$$
|\lam_{k,\ve}-\lam_{k,0}| \leq  C \ve^{\frac12-\tau}   k^\frac{2}{n-1} k^{2+\frac12 \frac{n}{n-1}}, \qquad k\in \N
$$
for any $\tau>0$, where $C$ is a computable positive constant  depending only on $\tau$, $\Omega$, $n$ and $\rho_\pm$. As stated in Conjecture \ref{conjetura}, we dare to guess that this bound should be improved to be at least $C \ve^\frac12 k^\frac{2}{n-1}$, and we let it as an open question. 

It is worth mentioning that this type of problems are closely related with study of continuity and strong continuity of eigenvalues. See for instance \cite{MZ, WYZ, WZ} and references therein.  

The Steklov boundary condition involved in our results brings on several technical problems with respect with the Dirichlet/Neumann case since  rapidly oscillating integrals on the boundary of the domain naturally appear. Dealing with these  integrals will be the main difficulty to face.  This task is overcame by means of the use of a duality approach via an auxiliary Neumann problem (Theorem \ref{teo.osci}); however, as it will be seen, our technique  requires precise estimates on eigenfunctions, therefore, with that end in Theorem \ref{teo.linfty} and Corollary \ref{cor.acotada} we compute explicit estimates on the $L^\infty$ norm of eigenfunctions to the Steklov problem \eqref{ec.stek.intro} via the Moser iteration argument introduced in \cite{BGT}.

As an application of Theorem \ref{main} we study the convergence of the first non-trivial curve of the Dancer-{F}u{\v{c}}{\'{\i}}k spectrum with Steklov boundary condition. Given two  sequences of function $\{a_\ve\}_{\ve>0}$ and $\{b_\ve\}_{\ve>0}$ satisfying \eqref{cond.rho} such that $a_\ve \cde a_0$, $b_\ve \cde b_0$ weakly* in $L^\infty(\partial\Omega)$ as $\ve\to0$, we consider the asymmetric eigenvalue problem given by
\begin{align} \label{eq.fucik.intro}   
\begin{cases}
-\Delta_p u_\ve + |u_\ve|^{p-2}u_\ve =0 &\text{ in }\Omega,\\
|\nabla u_\ve|^{p-2} \frac{\partial u_\ve}{\partial \nu}=   \alpha a_\ve(x)(u_\ve^+)^{p-1}- \beta b_\ve(x) (u_\ve^-)^{p-1} &\text{ on } \partial\Omega.
\end{cases}
\end{align}
The first non-trivial curve $\mathcal{C}_{a_\ve,b_\ve}$ of \eqref{eq.fucik.intro} is given by the set of $(\alpha,\beta)\in \R^+\times \R^+$ such that the corresponding eigenfunctions change their sign. For each $s\in \R^+$ let $(\alpha_\ve(s),\beta_\ve(s))$ be the intersection between $\mathcal{C}_{a_\ve,b_\ve}$ and line of slope $s$ starting at the origin of $\R^2$. Then, we parameterize the first nontrivial curve of \eqref{eq.fucik.intro} as $\{\alpha_\ve(s),\beta_\ve(s))\colon s>0\}$ and  denote it as $\mathcal{C}_{a_\ve,b_\ve}(s)$. For a fixed $\ve>0$, existence and properties on that curve were studied for instance in \cite{Anane}.
Similarly, denote by $\mathcal{C}_{a_0,b_0}(s)$ the first non-trivial curve corresponding to the limit problem 
\begin{align} \label{eq.fucik.intro.0}   
\begin{cases}
-\Delta_p u_0 + |u_0|^{p-2}u_0 =0 &\text{ in }\Omega,\\
|\nabla u_0|^{p-2} \frac{\partial u_0}{\partial \nu}=   \alpha_0 a_0(x)(u_0^+)^{p-1}- \beta_0 b_0(x) (u_0^-)^{p-1} &\text{ on } \partial\Omega
\end{cases}
\end{align}
obtained as $\ve\to 0$ in \eqref{eq.fucik.intro}. 

In Theorem \ref{teo2} we prove that $\mathcal{C}_{a_\ve,b_\ve}\to \mathcal{C}_{a_0,b_0}$ as $\ve\to 0$ in the sense that $\alpha_\ve(s) \to \alpha_0(s)$ and $\beta_\ve(s)\to \beta_0(s)$ as $\ve \to 0$ for each fixed $s>0$. In the case of periodic homogenization, i.e., when $a\ve(x)=a(\tfrac{x}{\ve})$ and $b_\ve(x)=b(\tfrac{x}{\ve})$ with $a,b$ two $Q-$periodic functions, being $Q$ the unit cube in $\R^n$, we further obtain that
$$
|\alpha_\ve(s)-\alpha_0(s)| \leq C \ve^{\frac{p-1}{p}-\tau} \max\{1,s^{-1}\}, \qquad 
|\beta_\ve(s)-\beta_0(s)| \leq C\ve^{\frac{p-1}{p}-\tau} \max\{1,s  \}.
$$
for each $\tau>0$, where $C$ is a positive constant depending of $n$, $p$, $\rho_\pm$, $\Omega$ and $\tau$.

Finally, we mention that the results stated in Theorems \ref{main} and \ref{teo2} still true when changing the $p-$Laplacian operator with a general quasilinear operator of the form $\diver(|A(x)\nabla u\cdot \nabla u|^\frac{p-2}{2} A(x)\nabla u)$, being $A$ a uniformly elliptic and symmetric matrix.

The article is organized as follows: in Section \ref{sec2} we introduce the notation on Sobolev spaces and functions of bounded variation used throughout the paper, as well as some remarks on the Steklov eigenvalue problem; in Section \ref{sec3} we deal with the study of rapidly oscillating integrals on the boundary; Section \ref{sec4} is devoted to obtain precise $L^\infty$ bounds of Steklov eigenfunction; in Section \ref{sec5} we provide for the proof of our main results; finally in Section \ref{sec6} we introduce some applications to asymmetric eigenvalue problems.

\section{Preliminaries} \label{sec2}
In this section we introduce some definitions and notation used in this paper.

\subsection{Oscillating functions}
Given a sequence of functions $\{\rho_\ve(x)\}_{\ve>0}$ satisfying condition \eqref{cond.rho} we denote by $\rho_0(x)$ its weak* limit in $L^\infty(\partial\Omega)$ as $\ve \to 0$. When we say that $\{\rho_\ve(x)\}_{\ve>0}$ is a $Q-$periodic sequence, being $Q$ the unit cube in $\R^n$,  we mean that $\rho_\ve(x):=\rho(\tfrac{x}{\ve})$ with $\rho(x+h)=\rho(x)$ for all $x\in \partial\Omega$ and $h\in \mathbb{Z}^n$; in this case we have that $\rho_0 \in \R$ and it is given by $\rho_0 = \int_{\mathbb{T}^n} \rho(x)\,dx$, being $\mathbb{T}^n$ the unit torus of $\R^n$.

\subsection{Sobolev spaces}
If $A\subset \R^n$ is an open set and $1\leq p\leq \infty$, we denote by $L^p(A)$ the usual space of $p-$summable functions on $A$ with norm  $\|\cdot \|_{L^p(\Omega)}$.  $W^{1,p}(\Omega)$ stand for the  Sobolev space of functions in $L^p(A)$ whose gradient in the sense of distributions belongs to $L^p(A,\R^n)$, endowed with the norm
$$
\|u\|_{W^{1,p}(\Omega)}^p:= \|u\|_{L^p(\Omega)}^p + \|\nabla u\|_{L^p(\Omega)}^p = \int_\Omega  |\nabla u|^p + |u|^p\,dx.
$$

We recall that if $\Omega\subset \R^n$ is an open bounded domain with Lipschitz boundary, there exists a continuous linear operator $T\colon W^{1,p}(\Omega)\to L^p(\partial\Omega)$ such that $Tu=u|_{\partial\Omega}$ if $\in W^{1,p}(\Omega)\cap C(\bar\Omega)$ and 
$$
\|Tu\|_{L^p(\partial\Omega)} \leq C_{\text{Tr}_p}(\Omega) \|u\|_{W^{1,p}(\Omega)}, \qquad \text{ for each } u\in W^{1,p}(\Omega).
$$

Given $A\subset \R^n$ we denote with $|A|$ its $n-$dimensional Lebesgue measure, and with $|\partial\Omega|$ its $(n-1)-$dimensional Hausdorff measure, which is referred as $\HH$. 

\subsection{Functions of bounded variation} 
If $A\subset \R^n$ is open, we say that $u\in BV(A)$ if $u\in L^1(A)$ and its derivative in the sense of distributions is a finite Radon measure on $A$, i.e., $Du\in \mathcal{M}_b(A;\R^n)$. The space of \emph{functions of bounded variation} on $A$ is denoted $BV(A)$, and it is a Banach space endowed with the norm 
$$
\|u\|_{BV(A)}:= \| u \|_{L^1(A)} + \|Du\|_{\mathcal{M}_b(A;\R^n)}.
$$
The quantity $|Du|(A):=\|Du\|_{\mathcal{M}_b(A;\R^n)}$ is the \emph{total variation} of $u$.

Notice that $W^{1,1}(A)\subset BV(A)$. Moreover, if $u\in W^{1,1}(A)$, then  $\|u\|_{BV(A)}=\|u\|_{W^{1,1}(A)}$.

We recall the Sobolev embedding in this setting: the space $BV(\R^n)$ is continuously embedded in $L^p(\R^n)$ for every $1\leq p\leq \frac{n}{n-1}$. 

If $\Omega\subset\R^n$ is an open bounded domain with Lipschitz boundary, there exists a continuous linear operator $T\colon BV(\Omega)\to L^1(\partial\Omega)$ such that, denoting $T(u)$ on $\partial\Omega$ still by $u$, the following integration by parts holds true for every $\varphi\in C_c^1(\R^n)$
$$
\int_\Omega u \frac{\partial \varphi}{\partial x_i}\,dx = \int_{\partial\Omega} u\varphi \nu_i \,d\HH - \int_\Omega \varphi \,d D_i u,
$$
where $\nu_i$ denotes the $i-$th component of the outer normal $\nu$. We denote $C_{\text{Tr}}(\Omega)$ the norm of $T$. Thanks to the last expression above, $T$ is a lifting to $BV(\Omega)$ of the trace operator on $W^{1,1}(\Omega)$, with the same norm. Moreover, the following result is a consequence of the last expression: if $u\in W^{1,1}(\Omega)$, then we have that $u 1_\Omega \in BV(\R^n)$ with 
$$
\|u 1_\Omega\|_{BV(\R^n)}=\int_\Omega |\nabla u| + |u|\,dx + \int_{\partial\Omega} |u| \,d\HH,
$$
where $1_\Omega(x)=1$ if $x\in \Omega$ and $1_\Omega(x)=0$ otherwise.

\subsection{The Steklov eigenvalue problem}

Given an open bounded domain $\Omega\subset \R^n$ with Lipschitz boundary and $\rho$ satisfying \eqref{cond.rho} we consider the following Steklov eigenvalue problem
\begin{align} \label{ec.stek}
\begin{cases}
-\Delta_p u + |u|^{p-2}u =0 &\text{ in }\Omega,\\
|\nabla u|^{p-2} \frac{\partial u}{\partial \nu}=   \lambda \rho |u|^{p-2}u &\text{ on } \partial\Omega,
\end{cases}
\end{align}
where $\frac{\partial u}{\partial \nu}$ is the outer normal derivative.

We say that $\lam>0$ is an \emph{eigenvalue} of \eqref{ec.stek} with \emph{eigenfunction} $ u\in W^{1,p}(\Omega)\setminus\{0\}$ if the following relation holds
\begin{equation} \label{aut}
\int_\Omega |\nabla u|^{p-2} \nabla u \cdot \nabla \varphi\,dx + \int_\Omega |u|^{p-2}u\varphi\,dx = \lam\int_{\partial\Omega} \rho |u|^{p-2}u\varphi d\HH \qquad \forall \varphi \in W^{1,p}(\Omega).
\end{equation}

Since $\Omega$ has Lipschitz boundary, it is well-know that \eqref{ec.stek} admits a sequence of variational eigenvalues $\{\lam_k(\rho)\}_{k\in\N}$ such that $0<\lam_1 \leq \lam_2 \leq \cdots \nearrow +\infty$
and they can be characterized by means of the following  minimax formula
\begin{equation} \label{minimax}
\lam_k(\rho) = \inf_{C\in \mathcal{C}_k} \sup_{v\in C} \frac{\int_\Omega  |\nabla v|^p + |v|^p\,dx}{\int_{\partial\Omega} \rho |v|^p\, d\HH} 
\end{equation}
where $\mathcal{C}_{k}=\{C\subset W^{1,p}(\Omega) \colon C \text{ compact}, C=-C, \gamma(C)\geq k\}$, where $\gamma(C)$ is the Krasnoselskii genus of $C$. 

In particular, the first and second eigenvalue admit the following characterization (see for instance \cite{FBR1})
\begin{equation} \label{autov.st}
\lam_1(\rho)=\min_{u\in W^{1,p}(\Omega)} \frac{\int_\Omega  |\nabla v|^p + |v|^p\,dx}{\int_{\partial\Omega} \rho |v|^p\, d\HH}, \qquad 
\lam_2(\rho)=\min_{u\in \mathcal{A}(\Omega)} \frac{\int_\Omega  |\nabla v|^p + |v|^p\,dx}{\int_{\partial\Omega} \rho |v|^p\, d\HH}
\end{equation}
where $A=\{u\in W^{1,p}(\Omega)\colon |\partial \Omega^\pm|\geq c \}$, being $\partial\Omega^+=\partial\Omega \cap \{u>0\}$, $\partial\Omega^-=\partial\Omega \cap \{u<0\}$ and $c=\rho_+ \lam_1 C_{\text{Tr}_p}(\Omega)^{p}$.

When $\rho\equiv 1$ we write $\lam_k$ instead of $\lam_k(1)$.  

Observe that in light of \eqref{cond.rho}, $\frac{1}{\rho_+} \lam_k  \leq \lam_k(\rho) \leq \frac{1}{\rho_-} \lam_k$ for every $k\in\N$. By using the isoperimetric inequality, 
\begin{equation*}
n\omega_n^\frac1n |\Omega|^\frac{n-1}{n}\leq  |\partial \Omega|
\end{equation*}
being $\omega_n$ the volume of the unit ball in $\R^n$, and testing with the the function $1$ in the previous characterization we get
\begin{equation} \label{cota.lam1}
\lam_1(\rho)\leq \frac{1}{\rho_-}  \frac{|\Omega|}{|\partial \Omega|} \leq c(n,\rho_-) |\partial\Omega|^\frac{1}{n-1}. 
\end{equation}

When $p=2$ it will be useful the following characterization of eigenvalues
\begin{equation} \label{caracp2}
\lam_{k}(\rho)= \min_{\mathcal{D}_k} \frac{\int_\Omega |\nabla v|^2 + |v|^2\,dx }{\int_{\partial\Omega} \rho |v|^2 \,d\HH}
\end{equation}
where
$$
\mathcal{D}_k = \{ v \in W^{1,2}(\Omega)\colon \int_{\partial\Omega} v v_j\, d\HH =0,  j=1,\ldots,k-1\}
$$
where $v_j$ is the corresponding $j-$th eigenfunction of $\lam_j(\rho)$.

Finally, we recall that the eigenvalues \eqref{caracp2} behave as follows:
\begin{equation} \label{crece}
c_1 \left(\frac{k}{|\partial\Omega|}\right)^\frac{1}{n-1} \leq  \lam_k(\rho) \leq c_2 \left(\frac{k}{|\partial\Omega|}\right)^\frac{1}{n-1}
\end{equation}
where $0<c_1<c_2<\infty$ are two constants independent of $k$ and $\Omega$. See for instance \cite{Pin}.

\section{Oscillatory integrals on the boundary} \label{sec3}

In this section we analyze the behavior of oscillating integrals on the boundary.

\begin{thm} \label{teo.osci}
Let $\Omega\subset \R^n$ be an open bounded convex set with Lipschitz boundary, $n\geq 3$.
Let $u\in W^{1,p}(\Omega)\cap L^\infty(\Omega)$,  $\ve \in (0,1)$ and $\rho$ a $Q-$periodic function satisfying \eqref{cond.rho} such that $\rho_\ve\cde\rho_0$ weakly* in $L^\infty(\partial\Omega)$ as $\ve\to 0$. Then, for every $\tau>0$ it holds that
\begin{equation} \label{lema.osc}
\left|\int_{\partial\Omega} (\rho_0 - \rho_\ve)|u|^p \,d\HH \right| \leq 
   C_1 C_\tau \ve^{\frac{p-1}{p}-\tau} \left( \|u\|^p_{L^p(\partial \Omega)}  +  \|u\|^{p-1}_{L^\infty(\Omega)} \|\nabla u\|_{L^p(\Omega)} \right),
\end{equation}
where $C_\tau>1$ is constant depending of $n$, $p$, $\tau$ and $\rho_\pm$ and $C_1>1$ is a constant depending of $\Omega$.
\end{thm}

\begin{proof}
Given $u\in W^{1,p}(\Omega)$ and $\ve\in(0,1)$ we can write
\begin{align*}
\int_{\partial\Omega} (\rho_0 - \rho_\ve)|u|^p \, d\HH &=
\int_{\partial\Omega} M_\ve |u|^p \,d\HH + \int_{\partial\Omega} (\rho_0 - \rho_\ve -M_\ve)|u|^p \,d\HH\\
&:=(i)+(ii),
\end{align*}
where 
$$
M_\ve:=\frac{1}{|\partial \Omega|}\int_{\partial\Omega} (\rho_0-\rho_\ve(x)) \,d\HH.
$$

Since $\Omega$ is convex and has Lipschitz boundary, by applying the principle of stationary phase (see \cite[Chapter VIII, Theorem 1]{Stein}), the following classical result in oscillatory integral theory  holds:
$$
|M_\ve|\leq c_1 \ve^\frac{n-1}{2}
$$
where the constant $c_1$ depends of $\rho_\pm$ and $\Omega$. Then 
$$
(i)\leq c_1  \ve^{\frac{n-1}{2}} \int_{\partial\Omega} |u|^p \,d\HH.
$$

In order to bound $(ii)$ we use a duality approach via  the following auxiliary Neumann problem: let $v_\ve\in W^{1,2}(\Omega)$ be the weak solution of
\begin{align*}
\begin{cases}
-\Delta v_\ve=0 &\text{ in }\Omega,\\
\nu \cdot \nabla v_\ve= b_\ve &\text{ on } \partial\Omega,
\end{cases}
\end{align*}
being $\nu$ the unit outward normal, where we have denoted $b_\ve(x):= \rho_0-\rho_\ve(x)- M_\ve$. The function $v_\ve$ satisfies that
$$
\int_\Omega     \nabla v_\ve \cdot \nabla \varphi\,dx = \int_{\partial\Omega} \nu \cdot    \nabla v_\ve \varphi\, d\HH =0 \qquad \forall \varphi \in C^\infty(\Omega)
$$
which, taking into account the boundary condition turns in
\begin{equation} \label{ec.auxi}
\int_\Omega   \nabla v_\ve \cdot \nabla \varphi\,dx = \int_{\partial\Omega} (\rho_0-\rho_\ve(x)- M_\ve) \varphi\, d\HH \qquad \forall \varphi\in C^\infty(\Omega).
\end{equation}
The behavior as $\ve \to 0$ of the solution $v_\ve$ to the previous problem  was studied recently in  \cite{ASS}. Indeed, since
$$
\int_{\partial\Omega} b_\ve\,d\HH = \int_{\partial\Omega} (\rho_0 - \rho_\ve(x) - M_\ve) \,d\HH  =0,
$$
in \cite[Theorem 5.3]{ASS} it is  proved that for each $\tau>0$ and $1\leq q <\infty$ the following gradient estimate holds
\begin{equation} \label{cota.grad}
\|\nabla v_\ve\|_{L^q(\Omega)}\leq c_\tau \ve^{\frac{1}{q} -\tau}.
\end{equation}
Therefore, testing \eqref{ec.auxi} with $\varphi=\frac{u^p}{p}$, $u\in W^{1,p}(\Omega)\cap L^\infty(\Omega)$,  and using Holder's inequality we get
\begin{align*}
(i) &\leq \left| \int_\Omega  u^{p-1}   \nabla v_\ve \cdot \nabla u\,dx \right| \leq   \|u\|^{p-1}_{L^\infty(\Omega)} \|\nabla v_\ve\|_{L^{p'}(\Omega)} \|\nabla u\|_{L^p(\Omega)}\\
&\leq c_\tau \ve^{\frac{1}{p'}-\tau} \|u\|^{p-1}_{L^\infty(\Omega)} \|\nabla u\|_{L^p(\Omega)}.
\end{align*}

Gathering the bounds for $(i)$ and $(ii)$ we get
$$
\left|\int_{\partial\Omega} (\rho_0 - \rho_\ve)|u|^p \,d\HH \right| \leq 
c_1  \ve^{\frac{n-1}{2}}  \|u\|^p_{L^p(\partial \Omega)} +  c_\tau \ve^{\frac{1}{p'}-\tau} \|u\|^{p-1}_{L^\infty(\Omega)} \|\nabla u\|_{L^p(\Omega)}.
$$
Observe that $\|u\|^p_{L^p(\partial \Omega)}$ in bounded in light of the Trace Theorem since $\Omega$ has Lipschitz boundary.

Since $\ve\in(0,1)$, $p>1$ and $n\geq 3$, from the last inequality we get that
$$
\left|\int_{\partial\Omega} (\rho_0 - \rho_\ve)|u|^p \,d\HH \right| \leq 
(1+c_1)(1+c_\tau)  \ve^{\frac{1}{p'}-\tau} ( \|u\|^p_{L^p(\partial \Omega)} +   \|u\|^{p-1}_{L^\infty(\Omega)} \|\nabla u\|_{L^p(\Omega)}).
$$
and  \eqref{lema.osc} follows.
\end{proof}

\section{Bounds of eigenfunctions in terms of eigenvalues} \label{sec4}
The aim of this section is to obtain precise bounds of the $L^\infty$ norm of eigenfunctions of \eqref{ec.stek} in terms of the corresponding eigenvalue.

\begin{thm} \label{teo.linfty}
Let $\Omega\subset \R^n$ be an open bounded set with Lipschitz boundary. 	For every eigenfunction $u\in W^{1,p}(\Omega)$, $p>1$ of the Steklov problem \eqref{ec.stek} with associated eigenvalue $\lam$ we have that $u\in L^\infty(\Omega)$ with
$$
\|u\|_{L^\infty(\Omega)}  \leq 
C  K(\lam, \Omega)\|u\|_{W^{1,p}(\Omega)}
$$
where $C$ is a positive constant depending of $n$, $p$ and $\rho_\pm$, and
$$
K(\lam,\Omega)=\max\left\{1, C_{\text{Tr}}(\Omega)^\frac{n}{p-1} , \left(\frac{\lam}{\lam_1}\right)^\frac{1}{p}\right\}   (\lam^\frac1p + |\Omega|^\frac{1}{np})^{\frac{n-1}{p-1}}    
$$
being $\lam_1$ the first eigenvalue of \eqref{ec.stek}.
\end{thm}

\begin{proof}
We split the proof in several steps for the reader's convenience. Along this proof $C$ will denote a positive constant depending only on $n$, $p$ and $\rho_+$, and it may change from line to line.

Let $u\in W^{1,p}(\Omega)$ be an eigenfunction of \eqref{ec.stek} corresponding to the  eigenvalue $\lam$.

{\bf Step 1:} We claim that if $|u|^\alpha \in W^{1,1}(\Omega)$ with $\alpha\geq p$, then
\begin{equation} \label{step1a}
|u|^{\alpha\chi}\in W^{1,1}(\Omega)
\end{equation}
with
\begin{equation} \label{step1aa}
\| |u|^{\alpha\chi}\|_{W^{1,1}(\Omega)} \leq C c(\Omega) \left( (1+\lam^\frac1p) \frac{\alpha\chi}{(\alpha-1)^\frac{1}{p}}  + |\Omega|^\frac{1}{np} \right) \| |u|^\alpha \|_{W^{1,1}(\Omega)}^\chi,
\end{equation}
where $c(\Omega)=1+C_{\text{Tr}}^{\frac1p+\frac{n}{p}\frac{p-1}{n-1}}$ and
\begin{equation}
\chi:=   \frac{np-1}{np-p}=\frac{p-1}{np-p}+1>1.
\end{equation}

Let us prove \eqref{step1a}. For every $M>0$ let $u_M:=(u\wedge M)\vee (-M)$ and $\varphi_M= |u_M|^{\alpha-p} u_M$.

Since $\varphi_M = F(u_M)$ with $u_M\in W^{1,p}(\Omega)\cap L^\infty(\Omega)$ and $F(s):=|s|^{\alpha-p}s \in C^1$, we get that $\varphi_M\in W^{1,p}(\Omega)$ with 
$$
\nabla \varphi_M= (\alpha-p)|u_M|^{\alpha-p} \nabla u_M.
$$
Testing \eqref{aut} with $\varphi=\varphi_M$, since $|u|^\alpha\in W^{1,1}(\Omega)$ and $W^{1,1}(\Omega)\subset BV(\Omega)$ we get
\begin{align*}
(\alpha-1) \int_\Omega |\nabla u|^{p-2 }\nabla u |u_M|^{\alpha-p} \nabla u_M\,dx &= \lam \int_{\partial\Omega} \rho |u|^{p-2} u  |u_M|^{\alpha-p} u_M \,d\HH\\
&\quad - \int_\Omega |u|^{p-2}u |u_M|^{\alpha-p} u_M\,dx\\
&\leq  \lam \rho_+ \int_{\partial\Omega} |u|^\alpha \, d\HH + \int_\Omega |u|\,dx\\
&\leq (1+\lam \rho_+ C_{\text{Tr}}) \| |u|^\alpha \|_{W^{1,1}(\Omega)}
\end{align*}
where we have used the Trace Theorem on $BV(\Omega)$.

Letting $M\to\infty$ and using the Monotone Convergence Theorem we get
\begin{align} \label{d1}
\begin{split}
(\alpha-1) \int_\Omega |\nabla u|^p |u|^{\alpha-p}\,dx& = (\alpha-1) \int_{\{-M<u<M\}} |\nabla u|^p |u_M|^{\alpha-p} \,dx\\
&\leq (1+\lam \rho_+ C_{\text{Tr}}) \| |u|^\alpha \|_{W^{1,1}(\Omega)}.
\end{split}
\end{align}

To see \eqref{step1a} let us prove first that $|u_M|^{\alpha\chi}\in W^{1,1}(\Omega)$. By using H\"older's inequality, \eqref{d1} and the explicit form of $\chi$ we obtain 
\begin{align*}
\int_\Omega |\nabla (|u_M|^{\alpha\chi})|\,dx &= \alpha\chi\int_\Omega |u_M|^{\alpha\chi-1}|\nabla u_M|\,dx\\
&\leq \alpha\chi\int_\Omega |u|^{\alpha\chi-\frac{\alpha}{p}} |u|^{\frac{\alpha}{p}-1} |\nabla u|\,dx\\
&\leq \frac{\alpha\chi}{(\alpha-1)^\frac{1}{p}} \left( \int_\Omega |u|^{p'\alpha\chi-\frac{\alpha p'}{p}} \right)^\frac{1}{p'} \left((\alpha-1)\int_\Omega |u|^{\alpha-p} |\nabla u|^p\,dx \right)^\frac1p\\
&\leq  (1+\lam^\frac1p \rho_+^\frac1p C_{\text{Tr}}^\frac1p ) \frac{\alpha\chi}{(\alpha-1)^\frac{1}{p}} \left( \int_\Omega |u|^\frac{\alpha n}{n-1} \right)^\frac{1}{p'}  \| |u|^\alpha \|_{W^{1,1}(\Omega)} ^\frac1p.
\end{align*}

Observe that the fact that $|u|^\alpha\in W^{1,1}(\Omega)$ together with the Lipschitz regularity of $\Omega$ yields $|u|^\alpha 1_\Omega \in BV(\R^n)$; hence by Sobolev's Embedding Theorem it follows that $|u|^\alpha 1_\Omega \in L^\frac{n}{n-1}(\R^n)$ with
\begin{equation} \label{sobo}
\left( \int_\Omega |u|^\frac{\alpha n}{n-1} \right)^\frac{n-1}{n}\leq C_n \| |u|^\alpha 1_\Omega   \|_{BV(\R^n)},
\end{equation}
where $C_n$ is a constant depending on $n$. Since
$$
\| |u|^\alpha 1_\Omega   \|_{BV(\R^n)} =\int_\Omega |u|^\alpha + \nabla (|u|^\alpha)\,dx + \int_{\partial\Omega} |u|^\alpha \, d\HH  \leq (1+C_{\text{Tr}})\| |u|^\alpha \|_{W^{1,1}(\Omega)},
$$
from the last three relations we finally obtain that
\begin{align*}
\int_\Omega |\nabla (|u_M|^{\alpha\chi})|\,dx &\leq 
(1+ \lam^\frac1p \rho_+^\frac1p C_{\text{Tr}}^\frac1p) \frac{\alpha\chi}{(\alpha-1)^\frac{1}{p}}  \times \\
&\quad \times \left(  C_n (1+C_{\text{Tr}}) \| |u|^\alpha \|_{W^{1,1}(\Omega)} \right)^\frac{n}{(n-1)p'}  \| |u|^\alpha \|_{W^{1,1}(\Omega)} ^\frac1p\\
 &\leq 
   C (1+\lam^\frac1p)   \tilde c(\Omega) \frac{\alpha\chi}{(\alpha-1)^\frac{1}{p}}  \| |u|^\alpha \|^\chi_{W^{1,1}(\Omega)}
\end{align*}
where $\tilde c(\Omega)=  (1+C_{\text{Tr}})^{\frac1p+\frac{n}{p}\frac{p-1}{n-1}}$.

Now, since $q=\frac{n}{(n-1)\chi} = \frac{np}{np-1}$, by using \eqref{sobo}
\begin{align*}
\int_\Omega |u_M|^{\alpha \chi}\,dx &\leq 
\left( \int_\Omega |u|^{\alpha \chi q}\,dx \right)^\frac1q |\Omega|^\frac{1}{q'}\\
&\leq
\left( \int_\Omega |u|^\frac{\alpha n}{n-1}\,dx \right)^\frac1q |\Omega|^\frac{1}{q'}
\leq
C \| |u|^\alpha \|_{W^{1,1}(\Omega)}^\chi |\Omega|^\frac{1}{np}.
\end{align*}

From the last two expressions it follows that
$$
\| |u_M|^{\alpha\chi}\|_{W^{1,1}(\Omega)} \leq C (1+\tilde c(\Omega)) \left( (1+\lam^\frac1p)  \frac{\alpha\chi}{(\alpha-1)^\frac{1}{p}}  + |\Omega|^\frac{1}{np} \right) \| |u|^\alpha \|_{W^{1,1}(\Omega)}^\chi.
$$
Letting $M\to\infty$ and using the Dominated Convergence Theorem we get \eqref{step1a} and
\begin{equation*}
\| |u|^{\alpha\chi}\|_{W^{1,1}(\Omega)} \leq C c(\Omega) \left( (1+\lam^\frac1p ) \frac{\alpha\chi}{(\alpha-1)^\frac{1}{p}}  + |\Omega|^\frac{1}{np} \right) \| |u|^\alpha \|_{W^{1,1}(\Omega)}^\chi
\end{equation*}
with $c(\Omega)=1+C_{\text{Tr}}^{\frac1p+\frac{n}{p}\frac{p-1}{n-1}}$, which gives \eqref{step1aa}.

\medskip

{\bf Step 2.} 
Let us prove that $u\in L^\infty(\Omega)$ with 
\begin{equation} \label{cota.paso2}
\|u\|_{L^\infty(\Omega)} \leq C c(\Omega)^{\frac{\gamma}{p}} (1+\lam^\frac1p + |\Omega|^\frac{1}{np})^{\frac{\gamma}{p}}      \| u^p \|_{W^{1,1}(\Omega)}^\frac1p.
\end{equation}
Let $\alpha:=p \chi^m$ with $m\in\N$, since $\chi>1$ we have that 
$$
\frac{p \chi^{m+1}}{(p\chi^m -1)^\frac1p} \leq p^\frac{p-1}{p} \chi^{m\frac{p-1}{p}+1}, 
$$
then, using that $\chi>1$ we can write \eqref{step1aa} as
\begin{align} \label{itera}
\begin{split} 
\| |u|^{p\chi^{m+1}}\|_{W^{1,1}(\Omega)} 
&\leq C c(\Omega) \left( (1+\lam^\frac1p)  \frac{p\chi^{m+1}}{(p\chi^{m}-1)^\frac{1}{p}}  + |\Omega|^\frac{1}{np} \right) \| |u|^{p\chi^{m}} \|_{W^{1,1}(\Omega)}^\chi\\
&\leq
C c(\Omega) \left( (1+\lam^\frac1p)  p^\frac{p-1}{p} \chi^{m\frac{p-1}{p}+1}  + |\Omega|^\frac{1}{np} \right) \| |u|^{p\chi^{m}} \|_{W^{1,1}(\Omega)}^\chi\\
&\leq
C c(\Omega) (1+\lam^\frac1p + |\Omega|^\frac{1}{np})  \chi^{m\frac{p-1}{p}+1}  \| |u|^{p\chi^{m}} \|_{W^{1,1}(\Omega)}^\chi.
\end{split}
\end{align}
 The previous relations gives
\begin{align*}
\| |u|^{p\chi^{m+1}}\|_{W^{1,1}(\Omega)}^{\chi^{-m-1}}
&\leq 
[C c(\Omega)]^{\chi^{-m-1}} (1+\lam^\frac1p + |\Omega|^\frac{1}{np})^{\chi^{-m-1}}  \times\\ &\quad \times (\chi^{m\frac{p-1}{p}+1})^{\chi^{-m-1}}  \| |u|^{p\chi^{m}} \|_{W^{1,1}(\Omega)}^{\chi^{-m}}.
\end{align*}
Iterating \eqref{itera} in the right-hand side of the above inequality we get
\begin{align} \label{desita}
\begin{split}
\| |u|^{p\chi^{m+1}}\|_{W^{1,1}(\Omega)}^{\chi^{-m-1}}
&\leq 
[C c(\Omega)]^{\gamma_m} (1+\lam^\frac1p + |\Omega|^\frac{1}{np})^{\gamma_m}  \chi^{\beta_m} \| u^p \|_{W^{1,1}(\Omega)}
\end{split}
\end{align}
where
$$
\gamma_m:=\sum_{j=1}^{m+1} \chi^{-j}, \qquad \beta_m:=\sum_{j=1}^m \left( j \frac{p-1}{p} +1\right) \chi^{-j-1}.
$$
Hence, the trace theorem in BV and \eqref{desita} yield
\begin{align*}
\left(\int_{\partial\Omega} |u|^{p\chi^{m+1}}d\HH \right)^{\chi^{-(m+1)}} &\leq C_{\text{Tr}}^{\chi^{-m-1}} [C c(\Omega)]^{\gamma_m} (1+\lam^\frac1p + |\Omega|^\frac{1}{np})^{\gamma_m}  \chi^{\beta_m} \| u^p \|_{W^{1,1}(\Omega)}.
\end{align*}
Finally, sending $m\to\infty$ and observing that
\begin{align*}
\gamma:=\lim_{m\to\infty}\gamma_m &= \frac{\chi^{-1}}{1-\chi^{-1}}=\frac{\chi^{-1}}{1-\chi^{-1}} = \frac{p(n-1)}{p-1}<\infty,\\
\beta:=\lim_{m\to\infty}\beta_m &= \frac{\chi(2p-1)-p}{p\chi(\chi-1)^2}=\frac{p(n-1)^2}{p-1} \frac{(n+1)p -1}{np-1}<\infty
\end{align*}
we obtain (remember that $\chi>1$)
$$
\|u^p\|_{L^\infty(\partial\Omega)} \leq [C c(\Omega)]^{\gamma} (1+\lam^\frac1p + |\Omega|^\frac{1}{np})^{\gamma}  \chi^{\beta} \| u^p \|_{W^{1,1}(\Omega)},
$$
from which
\begin{equation} \label{eq.paso3}
\|u\|_{L^\infty(\partial\Omega)} \leq 
C  c(\Omega)^{\frac{\gamma}{p}} (1+\lam^\frac1p + |\Omega|^\frac{1}{np})^{\frac{\gamma}{p}}    \| u^p \|_{W^{1,1}(\Omega)}^\frac1p.
\end{equation}
Finally,  \eqref{cota.paso2} follows from the last inequality and the maximum principle.

{\bf Step 3.} Testing in \eqref{aut} with $u$ we get
$
\int_\Omega |\nabla u|^p + |u|^p\,dx = \lam \int_{\partial\Omega} \rho |u|^p d\HH.
$
Hence, by using Young's inequality
\begin{align*}
\int_\Omega \nabla(u^p)\,dx &= p\int_\Omega u^{p-1} \nabla u\,dx \leq (p-1)\int_\Omega |u|^p\,dx + \int_\Omega |\nabla u|^p\,dx\\
&= (p-1)\int_\Omega |u|^p\,dx +  \lam \int_{\partial\Omega}\rho |u|^p d\HH
\end{align*}
from where
\begin{align} \label{p3.1}
\begin{split}
\|u^p \|_{W^{1,1}(\Omega)}^\frac1p &\leq \left(p\int_\Omega |u|^p\,dx +  \lam \rho_+ \int_{\partial\Omega}  |u|^p d\HH \right)^\frac1p\\
&\leq C\left(\|u\|_{L^p(\Omega)}^p +  \lam  C_{\text{Tr}_p}  \|u\|_{W^{1,p}(\Omega)}^p \right)^\frac1p\\
&\leq C \left(1 +  \lam^\frac1p \lam_1^{-\frac1p}   \right)\|u\|_{W^{1,p}(\Omega)}
\end{split}
\end{align}
where we have used that $C_{\text{Tr}_p} \leq \lam_1^{-1}$, being $\lam_1$ the first eigenvalue of \eqref{ec.stek} with $\rho\equiv 1$.

Moreover, observe that
\begin{align} \label{p3.2}
c(\Omega)^\frac{\gamma}{p}\leq (1 +C_{\text{Tr}}^{\frac{1}{p}\frac{n-1}{p-1}+\frac{n}{p}})^\frac{\gamma}{p} \leq C (1+C_{\text{Tr}}^\frac{n}{p-1}).
\end{align}

Finally, from  \eqref{eq.paso3}, \eqref{p3.1}, \eqref{p3.2} and the definitions of $c(\Omega)$ and $\gamma$ we get
\begin{align*}
\|u\|_{L^\infty(\Omega)} &\leq 
C  (1+C_{\text{Tr}}^\frac{n}{p-1})  (1+\lam^\frac1p + |\Omega|^\frac{1}{np})^{\frac{n-1}{p-1}}      \left(1 +  \lam^\frac1p \lam_1^{-\frac1p}   \right)\|u\|_{W^{1,p}(\Omega)}
\end{align*}
and the proof concludes.
\end{proof}

\begin{cor} \label{cor.acotada}
Let $u_1,u_2\in W^{1,p}(\Omega)$ be two eigenfunctions corresponding to the first and second eigenvalue of \eqref{ec.stek}, respectively. Then it holds that
$$
\|u_i\|_{L^\infty(\Omega)} \leq C K_i(\Omega)
\|u_i\|_{W^{1,p}(\Omega)}, \qquad i=1,2
$$
where $C$ is a positive constant depending of $n$, $p$ and $\rho_\pm$ and
\begin{align*}
K_1(\Omega)&=\left(\max\{ C_{\emph{Tr}}(\Omega)^n, |\partial\Omega|^{-\frac{1}{p}}, |\partial\Omega|^{\frac{1}{p}}\} \right)^\frac{1}{p-1},\\
K_2(\Omega)&=\left(\max\{ C_{\emph{Tr}}(\Omega)^n, |\partial\Omega|^{-\frac{p-1}{p}}, |\partial\Omega|^{\frac{1}{p}}\} \right)^\frac{1}{p-1}
\end{align*}
\end{cor}

\begin{proof}
Let $\lam=\lam_1(\rho)$ the first eigenvalue of \eqref{ec.stek}.

Observe that the isoperimetric inequality gives that $|\Omega|\leq C|\partial\Omega|^\frac{n}{n-1}$.
Testing with $\varphi=1$ in the definition of the first eigenvalue of \eqref{ec.stek} and using the mentioned inequality we get
$$
\lam_1(\rho) \leq \frac{|\Omega|}{\rho_-|\partial\Omega|} \leq C|\partial \Omega|^{-\frac{1}{n-1}}.
$$
These observations lead to
\begin{align*}
(\lam^\frac1p + |\Omega|^\frac{1}{np})^{\frac{n-1}{p-1}}   &\leq (|\partial\Omega|^{-\frac1p \frac{1}{n-1}} + |\partial\Omega|^{\frac1p\frac{1}{n-1}})^{\frac{n-1}{p-1}}  \\
&\leq C ( |\partial\Omega|^{-\frac{1}{p(p-1)}} + |\partial\Omega|^\frac{1}{p(p-1)})
\end{align*}
and the result for $u_1$ follows.

The estimate for $u_2$ follows analogously by using the bounds of $\lam_2(\rho)$ provided in \cite{Pin}.
\end{proof}

We conclude this section remaking that, since Theorem \ref{teo.linfty} gives that eigenfunction to \eqref{ec.stek} are bounded, from the regularity theory for solutions of degenerate elliptic equation proved by Lieberman in \cite{lieber}, the following result follows.

\begin{cor}
Let $u\in W^{1,p}(\Omega)\cap L^\infty(\Omega)$ be an eigenfunction to \eqref{ec.stek}. Then $u\in C^{1,\alpha}(\bar\Omega)$ for some $\alpha \in (0,1)$.
\end{cor}

\section{Convergence rates} \label{sec5}

In this section we prove our main results on convergence rates of eigenvalues.

\begin{thm} \label{main}
Let $\{\rho_\ve\}_{\ve>0}$ be a sequence of functions satisfying \eqref{cond.rho} such that $\rho_\ve \cde \rho_0$ weakly* in $L^\infty(\partial\Omega)$. Denote by $\lam_{k,\ve}$ and $\lam_{k,0}$ to the $k-$th variational eigenvalue of \eqref{problems}, respectively. Then
\begin{equation} \label{converge}
\lim_{\ve\to 0} \lam_{k,\ve} = \lam_{k,0}.
\end{equation}
When the sequence $\{\rho_\ve\}_{\ve>0}$ is further $Q-$periodic we have that for any $\tau>0$,
$$
|\lam_{k,\ve}-\lam_{k,0}|\leq c\cdot  C(\Omega) \ve^{\frac{p-1}{p}-\tau}, \qquad k=1,2,
$$
being $c$ a positive constant depending only of $\tau$, $n$, $p$ and $\rho_\pm$ and
$$
C(\Omega)=C_1 |\partial\Omega|^\frac{2(p-1)}{n-1} \max\{   C_{\emph{Tr}}(\Omega)^n, |\partial\Omega|^\frac1p, |\partial\Omega|^{-\frac1p}, |\partial\Omega|^{-\frac{1}{n-1}}, |\partial\Omega|^{-\frac{p-1}{p}}\}.
$$
with $C_1$ the constant of Theorem \ref{teo.osci}.
\end{thm}

\begin{proof}
Let us start by proving \eqref{converge}. Let $\lam_{k,\ve}$ be the $k-$th variational eigenvalue of $(P_{\rho_\ve})$ with corresponding eigenfunction $u_{k,\ve}\in W^{1,p}(\Omega)$. Similarly, let $\lam_{k,0}$ and $u_{k,0}\in W^{1,p}(\Omega)$ be the $k-$th variational eigenpair of $(P_{\rho_0})$. 

The proof of \eqref{converge} follows by the weak* convergence of the sequence $\rho_\ve$ as $\ve\to0$ together with the variational characterization of the eigenvalues. First, since $\rho_\ve\cde \rho_0$ weakly* in $L^\infty(\partial\Omega)$ as $\ve\to 0$, and $\Omega$ has Lipschitz boundary it follows that
$$
\lim_{\ve\to 0}\int_{\partial\Omega} (\rho_\ve -\rho_0)|u|^p\,dx = 0 \quad \text{ for all } u\in W^{1,p}(\Omega), 
$$
from where
\begin{equation*}
\frac{\int_{\partial\Omega} \rho_\ve |u|^p \,d\HH}{ \int_{\partial\Omega} \rho_0 |u|^p \,d\HH } =  \frac{\int_{\partial\Omega} \rho_0 |u|^p \,d\HH}{ \int_{\partial\Omega} \rho_\ve |u|^p \,d\HH } = 1 + o(1).
\end{equation*}
Therefore, taking $C_{k,\delta}\in \mathcal{C}_k$ such that 
$$
\lam_{k,0}=\sup_{u\in C_{k,\delta}} \frac{\|u\|_{W^{1,p}(\Omega)}^p}{\int_{\partial\Omega} \rho_0 |u|^p\,d\HH} + o(\delta)
$$
and testing with $C_{k,\delta}$ in the characterization of $\lam_{k,\ve}$ we get
\begin{align*}
\lam_{k,\ve}\leq \sup_{u\in C_{k,\delta}} \frac{\|u\|_{W^{1,p}(\Omega)}^p}{\int_{\partial\Omega} \rho_0 |u|^p\,d\HH} \frac{\int_{\partial\Omega} \rho_0 |u|^p\,d\HH}{\int_{\partial\Omega} \rho_\ve |u|^p\,d\HH}  = (\lam_{k,\ve}+o(\delta))(1+o(1)).
\end{align*}
Letting $\delta\to 0$ and $\ve\to 0$ we obtain that $\lim_{\ve\to 0} \lam_{k,\ve} \leq \lam_{k,0}$. Interchanging the roles of $\lam_\ve$ and $\lam_0$ it follows that $\lim_{\ve\to 0} \lam_{k,\ve} \geq \lam_{k,0}$, and then \eqref{converge} follows.
 
\medskip

Let us prove now the rates of convergence for the first two eigenvalues.

Let $\lam_{\ve}$ be the first eigenvalue of $(P_{\rho_\ve})$ with corresponding eigenfunction $u_{\ve}\in W^{1,p}(\Omega)$. Similarly, let $\lam_{0}$ and $u_{0}\in W^{1,p}(\Omega)$ be the first eigenpair of $(P_{\rho_0})$.
 
{\bf Step 1.}

First, observe that Theorem \ref{lema.osc} together with the Trace Theorem gives that
$$
\left|\int_{\partial\Omega} (\rho_0 - \rho_\ve)|u_\ve|^p \,d\HH \right| \leq 
  C \ve^{\frac{p-1}{p}-\tau}  \left( \lam_1^{-1} \|u_\ve\|_{W^{1,p}(\Omega)}^p  +  \|u_\ve\|^{p-1}_{L^\infty(\Omega)} \|u_\ve\|_{W^{1,p}(\Omega)} \right),
$$
where $C=C_1 C_\tau$ is a positive constant depending of $p$, $n$, $\tau$, $\rho_\pm$ and $\Omega$.

The previous relation can be bounded by using Corollary \ref{cor.acotada} as
\begin{equation} \label{eqa.1}
\left|\int_{\partial\Omega} (\rho_0 - \rho_\ve)|u_\ve|^p \,d\HH \right| \leq 
 C  \mathcal{K}  \ve^{\frac{p-1}{p}-\tau} \|u\|_{W^{1,p}(\Omega)}^p,
\end{equation}
where, being $K_1$ the constant of Corollary \ref{cor.acotada}, we denote
$$
\mathcal{K}(\Omega)=\lam_1^{-1} + K_1(\Omega)^{p-1}.
$$

From \eqref{eqa.1}, \eqref{cond.rho} and the variational characterization of $\lam_\ve$ it is deduced that
\begin{align} \label{rel.1}
\begin{split}
\frac{\int_{\partial\Omega}   \rho_\ve |u_\ve|^p \,d\HH }{\int_{\partial\Omega} \rho_0 |u_\ve|^p \,d\HH} &\leq 
1+    C  \mathcal{K}  \ve^{\frac{p-1}{p}-\tau} \frac{\|u_\ve\|_{W^{1,p}(\Omega)}^p}{\int_{\partial\Omega} \rho_0 |u_\ve|^p \,d\HH}\\
&\leq 
1+   C  \mathcal{K}  \ve^{\frac{p-1}{p}-\tau} \frac{\rho_+}{\rho_-} \frac{\|u_\ve\|_{W^{1,p}(\Omega)}^p}{\int_{\partial\Omega}   \rho_\ve |u_\ve|^p \,d\HH}\\
&\leq 
1+    C  \mathcal{K}  \ve^{\frac{p-1}{p}-\tau} \lam_\ve.
\end{split}
\end{align}
With similar reasoning it is deduced that
\begin{align} \label{rel.2}
\begin{split}
\frac{\int_{\partial\Omega} \rho_0 |u_0|^p \,d\HH }{\int_{\partial\Omega}  \rho_\ve |u_0|^p \,d\HH} \leq 
1+   C  \mathcal{K}  \ve^{\frac{p-1}{p}-\tau} \lam_0.
\end{split}
\end{align}

{\bf Step 2.}
In light of the variational characterization of $\lam_\ve$, testing with $u_0$ and using \eqref{rel.2} we obtain
\begin{align*}
\lam_\ve \leq \frac{\| u_0 \|_{W^{1,p}(\Omega)}^p}{\int_{\partial\Omega} \rho_\ve |u_0|^p\,d\HH}&= 
\frac{\| u_0 \|_{W^{1,p}(\Omega)}^p}{\int_{\partial\Omega} \rho_0 |u_0|^p\,d\HH} \frac{\int_{\partial\Omega} \rho_0 |u_0|^p \,d\HH }{\int_{\partial\Omega}  \rho_\ve |u_0|^p \,d\HH}\\
&\leq \lam_0\left( 1+    C  \mathcal{K}  \ve^{\frac{p-1}{p}-\tau}   \lam_0\right)
\end{align*}
from where
$$
\lam_\ve - \lam_0 \leq  C  \mathcal{K}  \ve^{\frac{p-1}{p}-\tau}   \lam_0^2.
$$
Similarly, using the variational characterization of $\lam_0$, testing with $u_\ve$ and invoking \eqref{rel.1} it is obtained that
$$
\lam_0 - \lam_\ve \leq   C  \mathcal{K} \ve^{\frac{p-1}{p}-\tau}   \lam_\ve^2.
$$

{\bf Step 3.}
From the last two relations describing the difference between $\lam_\ve$ and $\lam_0$ we finally get 
\begin{equation} \label{cota.final}
|\lam_\ve-\lam_0|\leq C   \mathcal{K} \ve^{\frac{p-1}{p}-\tau}   \max\{    \lam_\ve^2, \lam_0^2 \}.
\end{equation}
Observe that from \eqref{cond.rho}, $\max\{ \lam_\ve,\lam_0\}\leq \rho_-^{-1}\lam_1$. Then the desired bound follows from \eqref{cota.final} taking into account the definition of $\mathcal{K}$, $K_1$ and the bound \eqref{cota.lam1} for $\lam_1$.

The same analysis can be done to obtain a bound for the second variational eigenvalues of \eqref{problems} taking into account the variational characterization \eqref{autov.st} of these eigenvalues, the definition of the constant $K_2$ of Corollary \ref{cor.acotada} and expression \eqref{crece}.
\end{proof}

Next, we provide for the proof of the convergence rates for eigenvalues of Steklov eigenvalue of the Laplacian.

\begin{thm}\label{teo.p2}
Let $\{\rho_\ve\}_{\ve>0}$ be a sequence of $Q$-periodic functions satisfying \eqref{cond.rho} such that $\rho_\ve \cde \rho_0$ weakly* in $L^\infty(\partial\Omega)$. Denote by $\lam_{k,\ve}$ and $\lam_{k,0}$ to the $k-$th eigenvalue of \eqref{problems} with $p=2$, respectively. Then
$$
|\lam_{k,0}-\lam_{k,\ve}| \leq  c\cdot C(\Omega) \ve^{\frac12-\tau}  k^{\frac{2}{n-1}} k^{2+\frac12 \frac{n}{n-1}}  ,
$$
where $c$ a positive constant depending only of $\tau$, $n$ and $\rho_\pm$ and
$$
C(\Omega)=C_1\max\{1, C_{Tr}(\Omega)^n\} \max\{|\partial\Omega|^{-1}, |\partial\Omega| \}^\frac{1}{2(n-1)}
$$
with $C_1$ the constant of Theorem \ref{teo.osci}.

\end{thm}

\begin{proof}
For each $k\in\N$ denote by $u_{k,\ve}\in W^{1,2}(\Omega)$ and  $u_{k,0}\in W^{1,2}(\Omega)$ the eigenfunctions corresponding to $\lam_{k,\ve}$ and $\lam_{k,0}$, respectively.

{\bf Step 1}. Fixed $k\in \N$ and $\ve\in(0,1)$, we define  the test function
$$
\varphi = c_{1,\ve} u_{1,0} + \cdots + c_{k,\ve} u_{k,0} \in W^{1,2}(\Omega)
$$
where the constants $c_{k,\ve}$ are chosen such that
$$
\int_{\partial\Omega} \rho_\ve u_{j,\ve} \varphi \,d\HH = 0 \quad \forall j=1,\ldots, k-1.
$$ 
Then, in light of the variational characterization \eqref{caracp2} and the fact   the sequence of eigenvalues in non-decreasing, we get
\begin{align} \label{ecp2}
\begin{split}
\lam_{k,\ve} &\leq \frac{\displaystyle\int_\Omega |\nabla \varphi|^2 \,dx}{\displaystyle \int_{\partial\Omega} \rho_\ve \varphi^2\,d\HH}
 \leq\frac{ \displaystyle{\sum_{j=1}^k c_{j,\ve}^2 \lam_{j,0} \int_{\partial \Omega} \rho_0 u_{j,0}^2\,d\HH} }{ \displaystyle \int_{\partial\Omega} \rho_\ve \varphi^2 \,d\HH}
\leq 
\lam_{k,0} \frac{\displaystyle \int_{\partial\Omega} \rho_0 \varphi^2 \,d\HH }{\displaystyle \int_{\partial\Omega} \rho_\ve \varphi^2 \,d\HH}.
\end{split}
\end{align}

{\bf Step 2}.
Theorem \ref{lema.osc} and the Trace Theorem  give
$$
\left|\int_{\partial\Omega} (\rho_0 - \rho_\ve)\varphi^2 \,d\HH \right| \leq 
C_1 C_\tau \ve^{\frac12-\tau}  \left( \lam_1^{-1} \|\varphi\|_{W^{1,2}(\Omega)}^2  +  \|\varphi\|_{L^\infty(\Omega)} \|\varphi\|_{W^{1,2}(\Omega)} \right).
$$
By using Theorem \ref{teo.linfty} we can bound
\begin{align*}
\|\varphi\|_{L^\infty(\Omega)} \|\varphi\|_{W^{1,2}(\Omega)} 
&\leq
\left(\sum_{j=0}^k |c_{j,\ve}| \|u_{j,0}\|_{L^\infty(\Omega)} \right)   \|\varphi\|_{W^{1,2}(\Omega)} \\
&\leq
\left(\sum_{j=0}^k  C K(\lam_{j,0},\Omega) |c_{j,\ve}|  \|u_{j,0}\|_{W^{1,2}(\Omega)}  \right)  \|\varphi\|_{W^{1,2}(\Omega)} \\
&\leq
C\max_{1\leq j \leq k}K(\lam_{j,0},\Omega)\cdot  k \|\varphi\|_{W^{1,2}(\Omega)}^2
\end{align*}
from where
\begin{align*}
\frac{\int_{\partial\Omega}   \rho_0 \varphi^2 \,d\HH }{\int_{\partial\Omega} \rho_\ve \varphi^2 \,d\HH} &\leq 
1+    C_1 C_\tau  \ve^{\frac12-\tau} \mathcal{T}(k,\Omega) \frac{\|\varphi\|_{W^{1,2}(\Omega)}^2}{\int_{\partial\Omega} \rho_\ve \varphi^2 \,d\HH}
\end{align*}
where
$$
\mathcal{T}(k,\Omega):= \lam_1^{-1} + k \max_{1\leq j \leq k} K(\lam_{j,0},\Omega),
$$
here $C_\tau$ and $C_1$ are the constants of Theorem \ref{lema.osc}, $\lam_1$ is the first eigenvalue of \eqref{ec.stek} with $p=2$ and
$$
K(\lam_{j,0},\Omega)=\max\left\{1, C_{\text{Tr}}(\Omega))^n , \sqrt{\frac{\lam_{j,0}}{\lam_1}} \right\}   (\lam_{j,0}^\frac12 + |\Omega|^\frac{1}{2n})^{n-1}, \qquad j=1,\ldots,k.
$$

{\bf Step 3}. Estimate of $\mathcal{T}(k,\Omega)$.

First observe that for each $1\leq j \leq k$, in light of the isoperimetric inequality and \eqref{crece}, for positive fixed constants $c_1$ and $c_2$ independent of $k$ and $\Omega$ it holds that  $c_1 (\frac{j}{|\partial\Omega|})^\frac{1}{n-1} \leq  \lam_{j,0} \leq c_2 (\frac{j}{|\partial\Omega|})^\frac{1}{n-1}$, from where
\begin{align*}
(\lam_{j,0}^\frac12 + |\Omega|^\frac{1}{2n})^{n-1}  \leq C j^\frac12 \max\{|\partial\Omega|^{-\frac12}, |\partial\Omega|^\frac12 \}\\
\lam_1^{-1}\leq C |\partial \Omega|^\frac{1}{n-1}, \qquad 
\sqrt{\frac{\lam_{j,0}}{\lam_1}} \leq C j^\frac{1}{2(n-1)}
\end{align*}
giving that
\begin{align*}
K(\lam_{j,0},\Omega)\leq \max\{1, C_{Tr}(\Omega)^n\} \max\{|\partial\Omega|^{-\frac12}, |\partial\Omega|^\frac12 \} j^{\frac12 \frac{n}{n-1}},
\end{align*}
and consequently, since $1\leq j \leq k$ and using \eqref{cota.lam1}, 
$$
\mathcal{T}(k,\Omega)\leq C (1+|\partial\Omega|^\frac{1}{n-1}) \max\{1, C_{Tr}(\Omega)^n\} \max\{|\partial\Omega|^{-\frac12}, |\partial\Omega|^\frac12 \} k^{1+\frac12 \frac{n}{n-1}}.
$$

{\bf Step 4}. Computation of
$
\frac{\|\varphi\|_{W^{1,2}(\Omega)}^2}{\int_{\partial\Omega} \rho_\ve \varphi^2 \,d\HH}.
$

Since $u_{j,0}$ are eigenfunctions corresponding to $\lam_{j,0}$, $1\leq j \leq k$, and $\lam_{j,0}\leq \lam_{k,0}$ for each $1\leq j \leq k$,
\begin{align*}
\frac{\|\varphi\|_{W^{1,2}(\Omega)}^2}{\displaystyle \int_{\partial\Omega} \rho_\ve \varphi^2 \,d\HH} &\leq \frac{\displaystyle\sum_{j=0}^k \int_\Omega |c_{j,\ve}|^2( |\nabla u_{j,0}|^2 + |u_{j,0}|^2) \,dx}{\displaystyle \sum_{j=0}^k \int_{\partial\Omega} \rho_\ve |c_{j,\ve}|^2  |u_{j,0}|^2 \,d\HH }\\
&\leq
\frac{\rho_+}{\rho_-}\sum_{j=0}^k \frac{\displaystyle\int_\Omega  |\nabla u_{j,0}|^2 + |u_{j,0}|^2 \,dx}{\displaystyle  \int_{\partial\Omega} \rho_0  |u_{j,0}|^2 \,d\HH }\\
&\leq \frac{\rho_+}{\rho_-} \sum_{j=0}^k \lam_{j,0} \leq \frac{\rho_+}{\rho_-} k \lam_{k,0} \leq   C |\partial\Omega|^{-\frac{1}{n-1}} k^\frac{n}{n-1}
\end{align*}
where we have used \eqref{cond.rho} and \eqref{cota.lam1}.

{\bf Step 5}. Putting all together.

From Steps 1--4 and the bound of $\lam_{k,0}$ in \eqref{crece}
$$
\lam_{k,\ve}\leq \lam_{k,0} +  c\cdot C(\Omega)  \ve^{\frac12-\tau} k^{2+\frac12 \frac{n}{n-1} + \frac{2}{n-1}}
$$
where $c$ is a constant depending only $\tau$, $n$ and $\rho_\pm$ and
$$
C(\Omega)=C_1 C_\tau(1+|\partial\Omega|^{-\frac{1}{n-1}})\max\{1, C_{\text{Tr}}(\Omega)^n\} \max\{|\partial\Omega|^{-\frac12}, |\partial\Omega|^\frac12 \}.
$$

Interchanging the roles of $\lam_{k,\ve}$ and $\lam_{k,0}$, analogously it is obtained that
$$
\lam_{k,0}\leq \lam_{k,\ve} +    c\cdot C(\Omega) \ve^{\frac12-\tau}  k^{2+\frac12 \frac{n}{n-1} + \frac{2}{n-1}}
$$
from where
$$
|\lam_{k,0}-\lam_{k,\ve}| \leq  c\cdot C(\Omega) \ve^{\frac12-\tau}  k^{2+\frac12 \frac{n}{n-1} + \frac{2}{n-1}}
$$
and the proof finishes.
\end{proof}

\begin{conj} \label{conjetura}
In light of the results obtained for Dirichlet/Neumann eigenvalues, we conjecture that the bound of Theorem \ref{teo.p2} should be enhanced at least up to $C(\Omega) \sqrt{\ve} k^\frac{2}{n-1}$.
\end{conj}

\section{Dancer-{F}u{\v{c}}{\'{\i}}k spectrum} \label{sec6}
Given functions $a,b$ satisfying \eqref{cond.rho}, the Dancer-{F}u{\v{c}}{\'{\i}}k spectrum of the Steklov problem on $W^{1,p}(\Omega)$ is defined as the set $\Sigma=\Sigma(a,b)$ of those $(\alpha,\beta)\in \R^2$ such that
\begin{align} \label{eq.fucik}    \tag{$\mathcal{F}_{a,b}$}
\begin{cases}
-\Delta_p u + u =0 &\text{ in }\Omega,\\
|\nabla u|^{p-2} \frac{\partial u}{\partial \nu}=   \alpha a(x)(u^+)^{p-1}- \beta b(x) (u^-)^{p-1} &\text{ on } \partial\Omega
\end{cases}
\end{align}
has a nontrivial solution. We recall that $\lam_1(r)$ denotes the  the principal eigenvalue of \eqref{ec.stek} with with weight $\rho=r$. The Dancer-{F}u{\v{c}}{\'{\i}}k spectrum clearly contains the trivial lines $\mathcal{C}_{a,b}^{0,+}:=\{\lam_1(a)\}\times \R$ and 	$\mathcal{C}_{a,b}^{0,-}=\R\times\{\lam_1(b)\}$. It is proved in \cite{Anane} that there exists a first nontrivial curve $\mathcal{C}_{a,b}\subset \Sigma$ such that $\mathcal{C}_{a,b}^{0,+}$ and $\mathcal{C}_{a,b}^{0,-}$ are isolated in $\Sigma_{a,b}$ in the sense that there is not $(\alpha,\beta)$ solving \eqref{eq.fucik} between the trivial lines and $\mathcal{C}_{a,b}$.

Following the construction of \cite{CTV}, it can be deduced that the first nontrivial curve is characterized as
\begin{equation} \label{characte.fu}
\mathcal{C}_{a,b}(s):=\{ (\alpha(s), \beta(s))\colon s>0\}
\end{equation}
where $\alpha(s)=s^{-1}c(s)$, $\beta(s)=c(s)$, being
$$
c(s):=\inf_{(\omega_+,\omega_-) \in \mathcal{P}_2} \max\{s\lam_1(a,\omega_-), \lam_1(b,\omega_+)\}
$$
and
$$
\mathcal{P}_2=\{ (\omega_-,\omega_+)\subset \Omega \colon \omega_i \text{ is open and connected}, \omega_-\cap\omega_+ =\emptyset  \}.
$$
For every $s>0$ there exists $u\in W^{1,p}(\Omega)$ such that $(\{u^+>0\}, \{u^->0\})$ achieves $c(s)$. See \cite{Anane} for different characterizations of this curve.

The function $\alpha(s)$   is proved to be strictly decreasing and $\beta(s)$ strictly increasing; indeed, one has that $\alpha(s)\to \lam_1(a)$ as $s\to\infty$ and $\beta(s)\to\infty$ as $s\to\infty$. As a consequence, $\mathcal{C}_{a,b}$ is a strictly decreasing curve in $\R^2$ which is asymptotic to the lines $\mathcal{C}_{a,b}^{0,+}$ as $s\to \infty$ and to $\mathcal{C}_{a,b}^{0,-}$ as $s\to 0$.

Given two $Q-$periodic function  $a,b$ satisfying condition \eqref{cond.rho} we define $a_\ve(x):=a(\tfrac{x}{\ve})$ and $b_\ve(x):=b(\tfrac{x}{\ve})$. For each fixed $\ve>0$ we consider the asymmetric Steklov problem $(\mathcal{F}_{a_\ve,b_\ve})$ whose first nontrivial curve is given by
$$
\mathcal{C}_{a_\ve,b_\ve}(s)=\{(\alpha_\ve(s),\beta_\varepsilon(s)),\, s\in\R^+\} = \{(s^{-1}c_\ve(s) , c_\ve(s)),\, s\in \R^+\}\\
$$
where $c_\ve(s):=\inf_{(\omega_+,\omega_-)\in \mathcal{P}_2} \max\{s\lam_1(a_\ve,\omega_-),\lam_1(b_\ve,\omega_+)\}$. 

As $\ve\to 0$ we prove that $\mathcal{C}_{a_\ve,b_\ve}$ converges to a limit curve $\mathcal{C}_{a_0,b_0}$ given by
$$
\mathcal{C}_{a_0,b_0}(s):=\{(\alpha_0(s),\beta_0(s)),\, s\in\R^+\} = \{(s^{-1}c_0(s) , c_0(s)),\, s\in \R^+\}
$$
where $c_0(s)= \max\{s\lam_1(a_0,\omega_-),\lam_1(b_0,\omega_+)\}$. This limit curve is the first nontrivial curve of problem $(\mathcal{F}_{a_0,b_0})$, where $a_0$ and $b_0$ are the weak* limit in $L^\infty(\partial\Omega)$ of $a_\ve$ and $b_\ve$ as $\ve\to 0$, respectively, where 
$$
a_0=\int_{\mathbb{T}^n} a(x)\,dx, \quad b_0=\int_{\mathbb{T}^n} b(x)\,dx
$$
being $\mathbb{T}^n$ the unit torus of $\R^n$.

\begin{thm} \label{teo2}
With the previous notation in force, let $\mathcal{C}_{a_\ve,b_\ve}$ and $\mathcal{C}_{a_0,b_0}$ be the first non-trivial curves of problems $(\mathcal{F}_{a_\ve,b_\ve})$ and $(\mathcal{F}_{a_0,b_0})$, respectively. Then $\mathcal{C}_{a_\ve,b_\ve}\to \mathcal{C}_{a_0,b_0}$ as $\ve\to 0$ in the sense that
\begin{align*}
	|\beta_\ve(s)-\beta_0(s)|&=|c_\ve(s)-c_0(s)|\leq
	 C  \ve^{\frac{p-1}{p}-\tau} \max\{1,s  \},\\
	|\alpha_\ve(s)-\alpha_0(s)|&=s^{-1}|c_\ve(s)-c_0(s)|\leq
	 C \ve^{\frac{p-1}{p}-\tau} \max\{1,s^{-1}\}
\end{align*}
for each $\tau>0$, where $C$ is a positive constant depending of $n$, $p$, $\tau$, $\rho_\pm$ and $\Omega$.
\end{thm}
\begin{proof}
Let $s>0$ be fix. Let $(\omega_+,\omega_-)\in \mathcal{P}_2$ be a partition such that
$$
c_0(s)=\max\{s \lam_1(a_0,\omega_+), \lam_1(b_0,\omega_-)\}.	
$$
Testing with $(\omega_+,\omega_-)$ in the definition of $c_\ve(s)$ we get
\begin{align} \label{eq.c.0}
	\begin{split}
		c_\ve(s)\leq \max\{s\lam_1(a_\ve,\omega_+),\lam_1(b_\ve,\omega_-)\}.
	\end{split}
\end{align}
Now, Theorem \ref{main} allows us to bound $\lam_1(a_\ve,\omega_+)$ and $\lam_1(b_\ve,\omega_-)$ in terms of $\lam_1(a_0,\omega_+)$ and $\lam_1(b_0,\omega_-)$, from where, for each $\tau>0$ we get
\begin{align*}
\begin{split}
		c_\ve(s)&\leq \max\{s\big(\lam_1(a_0 ,\omega_+)+  C(\Omega) \ve^{\frac{p-1}{p}-\tau}
		\big),\lam_1(b_0,\omega_-) +  C(\Omega) \ve^{\frac{p-1}{p}-\tau}    \}\\
		&\leq  \max\{s\lam_1(a_0 ,\omega_+) ,\lam_1(b_0,\omega_-)\}+   C(\Omega) \ve^{\frac{p-1}{p}-\tau}  \max\{s,1\}\\
		&=  c_0(s) +C(\Omega) \ve^{\frac{p-1}{p}-\tau}   \max\{1,s  \}   		
\end{split}		
\end{align*}
where $C(\Omega)$ is a constant depending of $\tau$, $\Omega$, $n$, $p$ and $\rho_\pm$. Therefore we obtain that
$$
		c_\ve(s)\leq c_0(s)+  C(\Omega) \ve^{\frac{p-1}{p}-\tau} \max\{1,s  \}.
$$
Interchanging the roles of $c_\ve(s)$ and $c_0(s)$ we similarly obtain that
\begin{align} \label{eq.c.21}
		c_0(s)\leq c_\ve(s)+  C(\Omega) \ve^{\frac{p-1}{p}-\tau} \max\{1,s  \} 
\end{align}
which allows  to derive that
$$
	|c_\ve(s)-c_0(s)|\leq   C(\Omega) \ve^{\frac{p-1}{p}-\tau} \max\{1,s\} .
$$
Finally, the proof concludes taking in account the last inequality together with the definition of $\alpha_\ve, \alpha_0,  \beta_\ve$ and $\beta_0$.
\end{proof}

\section*{Acknowledgements}
This paper is partially supported by grants UBACyT 20020130100283BA, CONICET PIP 11220150100032CO and ANPCyT PICT 2012-0153.


\begin{thebibliography}{99}
\bibitem{ASS}
Aleksanyan, H., Shahgholian, H.,  Sj\"olin, P.,  \emph{Applications of Fourier Analysis in Homogenization of the Dirichlet Problem: $L^p$ Estimates}. Archive for Rational Mechanics and Analysis, 215(1), 65–87. (2014).

\bibitem{A}
Allaire, G., \emph{Shape optimization by the homogenization method}. Applied Mathematical Sciences, 146. Springer-Verlag, New York, 2002. xvi+456 pp. ISBN: 0-387-95298-5

\bibitem{Anane} 
Anane, A., Chakrone, O., Karim, B.,  Zerouali, A.,  \emph{The beginning of the {F}u{\v{c}}{\'{\i}}k spectrum for a Steklov Problem}. Boletim da Sociedade Paranaense de Matem\'atica, 27(1), 21-27. (2009).
 
\bibitem{BGT}
Bucur, D., Giacomini, A.,  Trebeschi, P.,  \emph{$L^\infty$ bounds of Steklov eigenfunctions and spectrum stability under domain variation}. preprint. (2019).

\bibitem{CZ}
Castro, C., Zuazua, E., \emph{High frequency asymptotic analysis of a string with rapidly
oscillating density}, European J. Appl. Math. 11 (2000), 595–622.

\bibitem{CK}
Choi, S.,  Kim, I. C.,  \emph{Homogenization for nonlinear PDEs in general domains with oscillatory Neumann boundary data}. Journal de Math\'ematiques Pures et Appliqu\'ees, 102(2), 419-448. (2014).
 

\bibitem{CTV} Conti, M., Terracini, S.,   Verzini, G.  \emph{On a class of optimal partition problems related to the {F}u{\v{c}}{\'{\i}}k spectrum and to the monotonicity formulae}. Calculus of Variations and Partial Differential Equations, 22(1), 45-72. (2005).

\bibitem{FBPS1} 
Fern\'andez Bonder, J., Pinasco, J.P., Salort, A., \emph{Eigenvalue homogenization problem with indefinite weights}. Bulletin of the Australian Mathematical Society, 93(1), 113-127. (2016).

\bibitem{FBPS2} 
Fern\'andez Bonder, J., Pinasco, J.P.,  Salort, A.  \emph{A Lyapunov type inequality for indefinite weights and eigenvalue homogenization}. Proceedings of the American Mathematical Society, 144(4), 1669-1680. (2016).

\bibitem{FBPS3} 
Fern\'andez Bonder, J., Pinasco, J.P.,  Salort, A. \emph{Eigenvalue homogenization for quasilinear elliptic equations with various boundary conditions}. Electronic Journal of Differential Equations, 2016(30), 1-15. (2016).

\bibitem{FBPS4} 
Fern\'andez Bonder, J., Pinasco, J.P.,  Salort, A. \emph{Homogenization of Fucik Eigencurves}. Mediterranean Journal of Mathematics, 14(2), 90. (2017)

\bibitem{FBR1}
Fern\'andez Bonder, J. , Rossi, J., \emph{A nonlinear eigenvalue problem with indefinite weights related to the Sobolev trace embedding}. Publicacions Matematiques, 221-235. (2002)

\bibitem{JKO}
Jikov, V., Kozlov, S., Oleinik, O., \emph{Homogenization of Differential Operators
and Integral Functionals}, Springer-Verlag, Berlin, 1994.

\bibitem{KLS}
Kenig, C., Lin, F., Shen, Z., \emph{Estimates of eigenvalues and eigenfunctions in periodic homogenization}. J. Eur. Math. Soc. (JEMS) 15 (2013), no. 5, 1901--1925.

\bibitem{k1}
Kesavan S., \emph{Homogenization of elliptic eigenvalue problems: part 1}, Appl. Math. Optim
5 (1979), 153–167.

\bibitem{k2}
Kesavan S., \emph{Homogenization of elliptic eigenvalue problems: part 2}, Appl. Math. Optim
5 (1979), 197–216.

\bibitem{lieber}
Lieberman, G., \emph{Boundary regularity for solutions of degenerate elliptic equations}, Nonlinear Anal. 12 (1988) 1203–1219.

\bibitem{MZ}
Meng, G.,  Zhang, M., \emph{Continuity in weak topology: First order linear systems of ODE}. Acta Mathematica Sinica, English Series, 26(7), 1287-1298. (2010).

\bibitem{MV}
Moskow, S., Vogelius, M., \emph{First order corrections to the homogenized eigenvalues of a periodic composite medium. The case of Neumann boundary conditions}, Preprint, Rutgers University (1997).

\bibitem{OSY}
Oleinik, O., Shamaev, A.,  Yosifian, G., \emph{Mathematical problems in elasticity and homogenization}. Elsevier, (2009).

\bibitem{Os}
Osborn, J., \emph{Spectral approximation for compact operators}, Math. Comp.,
29 (1975), 712-725.

\bibitem{Pin}
Pinasco, J. P., \emph{Asymptotic Behavior of the Steklov Eigenvalues For the $p-$Laplace Operator}. Advanced Nonlinear Studies, 7(3), 319-328.  (2007).

\bibitem{S1}
Salort, A.,  \emph{Precise homogenization rates for the {F}u{\v{c}}{\'{\i}}k spectrum}. Nonlinear Differential Equations and Applications NoDEA, 24(4), 1-17. (2017).

\bibitem{S2}
Salort, A.,  \emph{Convergence rates in a weighted {F}u{\v{c}}{\'{\i}}k problem}. Advanced Nonlinear Studies, 14(2), 427-443. (2014).

\bibitem{SV}
Santosa, F., Vogelius, M., \emph{First-order corrections to the homogenized eigenvalues of
a periodic composite medium}, SIAM J. Appl. Math. 53 (1993), 1636–1668.

\bibitem{Stein}
Stein, E.,  \emph{Harmonic analysis: real-variable methods, orthogonality, and oscillatory integrals}. Princeton Mathematical Series, vol. 43, Princeton University Press, Princeton, NJ, 1993, With the assistance of Timothy S. Murphy, Monographs in Harmonic Analysis, III. MR 1232192

\bibitem{WYZ}
Wen, Z., Yang, M.,  Zhang, M., \emph{Complete Continuity of Eigen-Pairs of Weighted Dirichlet Eigenvalue Problem}. Mediterranean Journal of Mathematics, 15(2), 73. (2018).

\bibitem{WZ}
Wen, Z.,   Zhou, L., \emph{Strong continuity of eigen-pairs of the Schr\"odinger operator with integrable potentials}. Journal of Mathematical Analysis and Applications, 482(1), 123518. (2020).

\end{thebibliography}
\end{document}